\documentclass[12pt]{article}
\usepackage{amsmath,amsthm,amsfonts,amssymb}
\usepackage[pdftex,pdfborder={0 0 0}]{hyperref}
\usepackage{fullpage}
\usepackage{multirow}
\usepackage{array}
\usepackage{bbm}
\usepackage[noblocks]{authblk}
\usepackage{verbatim}
\usepackage{tikz}
\usepackage{float}
\usepackage{enumerate}
\usepackage{diagbox}
\usepackage{graphicx}
\usepackage{colortbl}
\usepackage{hhline}
\usepackage[makeroom]{cancel}

\newtheorem{theorem}{Theorem}[section]
\newtheorem{lemma}[theorem]{Lemma}

\newtheorem{corollary}[theorem]{Corollary}
\newtheorem{conjecture}[theorem]{Conjecture}

\theoremstyle{definition}

\newlength{\Oldarrayrulewidth}

\newcommand{\N}{\mathbb{N}}

\newcommand{\Z}{\mathbb{Z}}

\usepackage{tikz}
\newcommand*\circled[1]{\tikz[baseline=(char.base)]{
\node[shape=circle,draw,inner sep=1pt] (char) {#1};}}

\begin{document}

\title{An Investigation on Partitions with Equal Products}

\author[1]{Byungchul Cha\thanks{byungchulcha@muhlenberg.edu}}
\affil[1]{Department of Mathematics, Muhlenberg College}

\author[1]{Adam Claman\thanks{acclaman@muhlenberg.edu}}

\author[2]{Joshua Harrington\thanks{joshua.harrington@cedarcrest.edu}}
\affil[2]{Department of Mathematics, Cedar Crest College}

\author[3]{Ziyu Liu\thanks{liu35z@mtholyoke.edu}}
\affil[3]{Department of Mathematics, Mount Holyoke College}

\author[4]{Barbara Maldonado\thanks{barbara.maldonado23230@gmail.com}}
\affil[4]{Department of Mathematics, University of Houston}

\author[5]{Alexander Miller\thanks{amill971@live.kutztown.edu}}
\affil[5]{Department of Mathematics, Kutztown University of Pennsylvania}

\author[1]{Ann Palma\thanks{annpalma@muhlenberg.edu}}

\author[5]{Tony W.\ H.\ Wong\thanks{wong@kutztown.edu}}

\author[6]{Hongkwon (Vin) Yi\thanks{321\_vin@berkeley.edu}}
\affil[6]{Department of Mathematics, University of California, Berkeley}

\date{\today}

\maketitle

\begin{abstract}
An ordered triple $(s,p,n)$ is called admissible if there exist two different multisets $X=\{x_1,x_2,\dotsc,x_n\}$ and $Y=\{y_1,y_2,\dotsc,y_n\}$ such that $X$ and $Y$ share the same sum $s$, the same product $p$, and the same size $n$. We first count the number of $n$ such that $(s,p,n)$ are admissible for a fixed $s$. We also fully characterize the values $p$ such that $(s,p,n)$ is admissible. Finally, we consider the situation where $r$ different multisets are needed, instead of just two. This project is also related to John Conway's wizard puzzle from the 1960s.
\end{abstract}

\section{Introduction}
A multiset $X=\{x_1,x_2,\ldots,x_n\}$ of $n$ positive integers is an \emph{$n$-partition} of the sum $s=x_1+x_2+\cdots+x_n$.  Define the function $T\{x_1,x_2,\ldots,x_n\}=(s,p,n)$, where $p=x_1x_2\cdots x_n$.  Throughout this article, we will call $s$ and $p$ the sum and the product of the partition, respectively.  Our main focus will be on ordered triples $(s,p,n)$ for which there are at least two different $n$-partitions sharing the same sum $s$ and the same product $p$.  We call such ordered triples \emph{admissible}.  A positive integer $s$ is \emph{sum-admissible} if there exist integers $p$ and $n$ such that $(s,p,n)$ is admissible; similarly, a positive integer $p$ is \emph{product-admissible} if there exist integers $s$ and $n$ such that $(s,p,n)$ is admissible.  

For each integer $r\geq2$ and $n\geq3$, let $s_r(n)$ be the smallest positive integer, if it exists, such that for all integers $s\geq s_r(n)$, there are at least $r$ different $n$-partitions of $s$, namely $\{x_{i1},x_{i2},\dotsc,x_{in}\}$, where $i=1,2,\dotsc,r$, satisfying
\begin{enumerate}[\indent$(a)$]
\item\label{distinct} $x_{ij}\neq x_{i'j'}$ if $(i,j)\neq(i',j')$, and
\item\label{samespn} there exists $p\in\N$ such that for all $i=1,2,\dotsc,r$, $T\{x_{i1},x_{i2},\dotsc,x_{in}\}=(s,p,n)$.
\end{enumerate}
If condition $(\ref{distinct})$ is removed, then let $s_r^*(n)$ be the smallest positive integer, if it exists, such that for all integers $s\geq s_r^*(n)$, there are at least $r$ different $n$-partitions of $s$ satisfying only condition $(\ref{samespn})$. The following theorem is proved by John B. Kelly in 1964.

\begin{theorem}[$\cite{kelly1}$]\label{kelly}
For every integer $n\geq3$, $s_{n-1}(n)$ and $s_{n-1}^*(n)$ exist. Furthermore, $s_2(3)=23$ and $s_2^*(3)=19$.
\end{theorem}

In the same paper, Kelly mentioned that the only known values of $s_{n-1}(n)$ and $s_{n-1}^*(n)$ were when $n=3$, and all other values were unknown. He later showed that for $n\geq 3$ and for any positive integer $r$, there exist infinitely many integers $s$ for which there are $r$ mutually disjoint $n$-partitions of $s$ such that the products of the partitions are all equal \cite{kelly2}.  

According to Kelly, his investigation into $n$-partitions of equal sum and product began with a conjecture, communicated orally, of T. S. Motzkin. Motzkin conjectured that for each sufficiently large $s$, there exists a positive integer $p$ such that the triple $(s,p,3)$ is admissible.  Although Theorem \ref{kelly} proves and generalizes Motzkin's conjecture, there are still many curious open questions related to this conjecture.  In fact, just recently in 2015, Sadek and El-Sissi parameterized all admissible triples of the form $(s,p,3)$ \cite{sadek}.

In this article, our study of admissible triples is threefold.  First, in Section~\ref{secf(s)}, we determine the value of the function
$$f(s)=|\{n\in\mathbb{N}:(s,p,n)\text{ is admissible for some }p\in\mathbb{N}\}|$$
for each positive integer $s$. Second, in Section~\ref{secadmisprod} we provide a full characterization of product-admissible numbers.  We further prove in Section~\ref{secadmisprod} that if $q$ is a prime and $j$ is a positive integer, then $q^j$ is product-admissible if and only if $j\geq 2q+4$.  Third, in Section~\ref{secmultipleparts}, we provide an algorithm to effectively calculate the values of $s_r^*(n)$, and as a generalization of Kelly's results, we generate a list of values of $s_{n-1}^*(n)$ for $3\leq n\leq21$.  Finally, we end our article in Section \ref{conclusion} with several conjectures regarding $s_r^*(n)$.

Before we move on to the next section, we would like to mention that the problem of integer partitions with equal products has connections with the ``Conway's wizard problem \cite{bennett}." In the 1960's, John Conway posed the following riddle.

\begin{quotation}
\noindent Last night I sat behind two wizards on a bus and overheard the following:\\

\begin{tabular}{p{80pt}p{300pt}}
Blue Wizard:& I have a positive integer number of children, whose ages are positive integers. The sum of their ages is the number of this bus, while the product is my own age.\\
Red Wizard:& How interesting! Perhaps if you told me your age and the number of your children, I could work out their individual ages?\\
Blue Wizard:& No, you could not.\\
Red Wizard:& Aha! At last, I know how old you are!\\
\end{tabular}\\

\noindent Apparently the Red Wizard had been trying to determine the Blue Wizard's age for some time. Now, what was the number of the bus?
\end{quotation}

Solving this riddle is equivalent to finding a positive integer $s$ such that there is a unique product $p$ and an integer $n$ to produce an admissible triple $(s,p,n)$.

\section{The function $f(s)$}\label{secf(s)}
The following theorem is the main result of this section. For a fixed $s$, we count the number of positive integers $n$ such that $(s,p,n)$ is admissible for some product $p$.

\begin{theorem}\label{sequencefs}
When $1\leq s\leq 11$, $f(s)=0$, and when $s\geq19$, $f(s)=s-10$. Finally, $(f(s))_{s=12}^{18}=(1, 2, 4, 4, 6, 7, 7)$.
\end{theorem}

In order to prove this theorem, we first introduce several lemmas regarding the function $F(s)=\{n\in\mathbb{N}:(s,p,n)\text{ is admissible for some }p\in\mathbb{N}\}$.

\begin{lemma}\label{impossiblen}
For each $s\in\N$, $\{1,2,s-7,s-6,\dotsc,s-1,s\}\cap F(s)=\emptyset$. In other words, if $n\in\{1,2,s-7,s-6,\dotsc,s-1,s\}$, then $(s,p,n)$ is not admissible for any $p\in\mathbb{N}$.
\end{lemma}

\begin{proof}
If $n=1$, then the only $1$-partition of $s$ is $\{s\}$. If $n=2$, then all the $2$-partitions of $s$ are of the form $\{r,s-r\}$, where $r\in\N$. Assume that there exist two different partitions $\{r,s-r\}$ and $\{r',s-r'\}$ satisfying $T\{r,s-r\}=T\{r',s-r'\}=(s,p,2)$ for some $p\in\N$. Then $r(s-r)=r'(s-r')$, which implies $s(r-r')-(r^2-r'^2)=(s-r-r')(r-r')=0$. In other words, $r=r'$ or $r=s-r'$, contradicting that $\{r,s-r\}\neq\{r',s-r'\}$.

If $n=s-t$ for some $t=0,1,2,\dotsc,7$, we can now assume that $s-t\geq3$ for a meaningful discussion. Here is a table of partitions of $s$ into $s-t$ parts for each value $t$, together with their corresponding product $p$.

\begin{center}
\begin{tabular}{|c|l|c|}
\hline
$t$& Partitions of $s$ into $s-t$ parts& $p$\\
\hline
$0$& $\{1,1,\dotsc,1\}$& $1$\\
\hline
$1$& $\{1,1,\dotsc,1,2\}$& $2$\\
\hline
$2$& $\{1,1,\dotsc,1,3\}$& $3$\\
& $\{1,1,\dotsc,1,2,2\}$& $4$\\
\hline
$3$& $\{1,1,\dotsc,1,4\}$& $4$\\
& $\{1,1,\dotsc,1,2,3\}$& $6$\\
& $\{1,1,\dotsc,1,2,2,2\}$& $8$\\
\hline
$4$& $\{1,1,\dotsc,1,5\}$& $5$\\
& $\{1,1,\dotsc,1,2,4\}$& $8$\\
& $\{1,1,\dotsc,1,3,3\}$& $9$\\
& $\{1,1,\dotsc,1,2,2,3\}$& $12$\\
& $\{1,1,\dotsc,1,2,2,2,2\}$& $16$\\
\hline
$5$& $\{1,1,\dotsc,1,6\}$& $6$\\
& $\{1,1,\dotsc,1,2,5\}$& $10$\\
& $\{1,1,\dotsc,1,3,4\}$& $12$\\
& $\{1,1,\dotsc,1,2,2,4\}$& $16$\\
& $\{1,1,\dotsc,1,2,3,3\}$& $18$\\
& $\{1,1,\dotsc,1,2,2,2,3\}$& $24$\\
& $\{1,1,\dotsc,1,2,2,2,2,2\}$& $32$\\
\hline
\end{tabular}
\begin{tabular}{|c|l|c|}
\hline
$t$& Partitions of $s$ into $s-t$ parts& $p$\\
\hline
$6$& $\{1,1,\dotsc,1,7\}$& $7$\\
& $\{1,1,\dotsc,1,2,6\}$& $12$\\
& $\{1,1,\dotsc,1,3,5\}$& $15$\\
& $\{1,1,\dotsc,1,4,4\}$& $16$\\
& $\{1,1,\dotsc,1,2,2,5\}$& $20$\\
& $\{1,1,\dotsc,1,2,3,4\}$& $24$\\
& $\{1,1,\dotsc,1,3,3,3\}$& $27$\\
& $\{1,1,\dotsc,1,2,2,2,4\}$& $32$\\
& $\{1,1,\dotsc,1,2,2,3,3\}$& $36$\\
& $\{1,1,\dotsc,1,2,2,2,2,3\}$& $48$\\
& $\{1,1,\dotsc,1,2,2,2,2,2,2\}$& $64$\\
\hline
$7$& $\{1,1,\dotsc,1,8\}$& $8$\\
& $\{1,1,\dotsc,1,2,7\}$& $14$\\
& $\{1,1,\dotsc,1,3,6\}$& $18$\\
& $\{1,1,\dotsc,1,4,5\}$& $20$\\
& $\{1,1,\dotsc,1,2,2,6\}$& $24$\\
& $\{1,1,\dotsc,1,2,3,5\}$& $30$\\
& $\{1,1,\dotsc,1,2,4,4\}$& $32$\\
& $\{1,1,\dotsc,1,3,3,4\}$& $36$\\
& $\{1,1,\dotsc,1,2,2,2,5\}$& $40$\\
& $\{1,1,\dotsc,1,2,2,3,4\}$& $48$\\
& $\{1,1,\dotsc,1,2,3,3,3\}$& $54$\\
& $\{1,1,\dotsc,1,2,2,2,2,4\}$& $64$\\
& $\{1,1,\dotsc,1,2,2,2,3,3\}$& $72$\\
& $\{1,1,\dotsc,1,2,2,2,2,2,3\}$& $96$\\
& $\{1,1,\dotsc,1,2,2,2,2,2,2,2\}$& $128$\\
\hline
\end{tabular}
\end{center}

From this table, we can see that all products $p$ are unique for each $n=s-t$. Therefore, $1,2,s-7,s-6,\dotsc,s-1,s\notin F(s)$.
\end{proof}

\begin{lemma}\label{inductionstep}
Let $s\in\N$. If there exists a positive integer $n$ such that $n\in F(s)$, then for all positive integers $s'$ and $n'$ satisfying $n'\leq s'$, we have $n+n'\in F(s+s')$.
\end{lemma}

\begin{proof}
Suppose $n\in F(s)$. Then there exist at least two different multisets of $n$ positive integers, $\{x_1,x_2,\dotsc,x_n\}$ and $\{y_1,y_2,\dotsc,y_n\}$, satisfying $T\{x_1,x_2,\dotsc,x_n\}=T\{y_1,y_2,\dotsc,y_n\}=(s,p,n)$. For all positive integers $s'$ and $n'$ satisfying $n'\leq s'$, let $x_{n+1}=x_{n+2}=\dotsb=x_{n+n'-1}=y_{n+1}=y_{n+2}=\dotsb=y_{n+n'-1}=1$ and $x_{n+n'}=y_{n+n'}=s'-(n'-1)$. We can extend our multisets of $n$ positive integers to $\{x_1,x_2,\dotsc,x_n,x_{n+1},\dotsc x_{n+n'}\}$ and $\{y_1,y_2,\dotsc,y_n,y_{n+1},\dotsc,y_{n+n'}\}$ such that
$$T\{x_1,x_2,\dotsc,x_n,\dotsc,x_{n+n'}\}=T\{y_1,y_2,\dotsc,y_n,\dotsc,y_{n+n'} \}=(s+s',p(s'-(n'-1)),n+n').$$
This implies $n+n'\in F(s+s')$.
\end{proof}

\begin{lemma}\label{computationresults}
\begin{enumerate}[$(a)$]
\item\label{notinFs} For $s=11,12,15,18$, $3\notin F(s)$. Also, $4\notin F(13)$.
\item\label{inFs} For $s=13,14,16,17$, $3\in F(s)$. Also, $4\in F(12)$.
\end{enumerate}
\end{lemma}

\begin{proof}
Statement $(\ref{notinFs})$ is proved by exhaustion of all $3$-partitions of $11$, $12$, $15$, and $18$, as well as all $4$-partitions of $13$. Statement $(\ref{inFs})$ is due to the following observations.
\begin{center}
$T\{1,6,6\}=T\{2,2,9\}=(13,36,3)$ implies $3\in F(13)$,\\
$T\{1,5,8\}=T\{2,2,10\}=(14,40,3)$ implies $3\in F(14)$,\\
$T\{2,5,9\}=T\{3,3,10\}=(16,90,3)$ implies $3\in F(16)$,\\
$T\{3,6,8\}=T\{4,4,9\}=(17,144,3)$ implies $3\in F(17)$,\\
and $T\{1,3,4,4\}=T\{2,2,2,6\}=(12,48,4)$ implies $4\in F(12)$.
\end{center}
\end{proof}

Now, we are ready to prove Theorem \ref{sequencefs}.

\begin{proof}[Proof of Theorem $\ref{sequencefs}$]
When $1\leq s\leq 10$, for all positive integers $n\leq s$, $n\in\{1,2,s-7,s-6,\dotsc,s-1,s\}$. By Lemma \ref{impossiblen}, $f(s)=0$. When $s\geq11$, we summarize the procedures in the following table.
\begin{center}\begin{tabular}{|c||c|c|c|c|c|c|c|c|c|c|c|c|c|c|c|c|c|c|c|c|}
\hline
\textbf{Sum $s$}&\multicolumn{20}{c|}{\textbf{Values $n\leq s$}}\\
\hhline{|=||====================|}
11&1&2&\xcancel{3}&4&5&6&7&8&9&10&11&&&&&&&&&\\
\hline
12&1&2&\xcancel{3}&\circled{4}&5&6&7&8&9&10&11&12&&&&&&&&\\
\hline
13&1&2&\circled{3}&\xcancel{4}&\cellcolor{lightgray!50}5&6&7&8&9&10&11&12&13&&&&&&&\\
\hline
14&1&2&\circled{3}&\cellcolor{lightgray!50}4&\cellcolor{lightgray!50}5&\cellcolor{lightgray!50}6&7&8&9&10&11&12&13&14&&&&&&\\
\hline
15&1&2&\xcancel{3}&\cellcolor{lightgray!50}4&\cellcolor{lightgray!50}5&\cellcolor{lightgray!50}6&\cellcolor{lightgray!50}7&8&9&10&11&12&13&14&15&&&&&\\
\hline
16&1&2&\circled{3}&\cellcolor{lightgray!50}4&\cellcolor{lightgray!50}5&\cellcolor{lightgray!50}6&\cellcolor{lightgray!50}7&\cellcolor{lightgray!50}8&9&10&11&12&13&14&15&16&&&&\\
\hline
17&1&2&\circled{3}&\cellcolor{lightgray!50}4&\cellcolor{lightgray!50}5&\cellcolor{lightgray!50}6&\cellcolor{lightgray!50}7&\cellcolor{lightgray!50}8&\cellcolor{lightgray!50}9&10&11&12&13&14&15&16&17&&&\\
\hline
18&1&2&\xcancel{3}&\cellcolor{lightgray!50}4&\cellcolor{lightgray!50}5&\cellcolor{lightgray!50}6&\cellcolor{lightgray!50}7&\cellcolor{lightgray!50}8&\cellcolor{lightgray!50}9&\cellcolor{lightgray!50}10&11&12&13&14&15&16&17&18&&\\
\hline
19&1&2&\boxed{3}&\cellcolor{lightgray!50}4&\cellcolor{lightgray!50}5&\cellcolor{lightgray!50}6&\cellcolor{lightgray!50}7&\cellcolor{lightgray!50}8&\cellcolor{lightgray!50}9&\cellcolor{lightgray!50}10&\cellcolor{lightgray!50}11&12&13&14&15&16&17&18&19&\\
\hline
20&1&2&\boxed{3}&\cellcolor{lightgray!50}4&\cellcolor{lightgray!50}5&\cellcolor{lightgray!50}6&\cellcolor{lightgray!50}7&\cellcolor{lightgray!50}8&\cellcolor{lightgray!50}9&\cellcolor{lightgray!50}10&\cellcolor{lightgray!50}11&\cellcolor{lightgray!50}12&13&14&15&16&17&18&19&20\\
\hline
\end{tabular}\end{center}

In this table, those crossed-out entries, i.e., $\xcancel{n}$, indicate $n\notin F(s)$ by Lemma $\ref{computationresults}(\ref{notinFs})$. Those circled entries, i.e., $\circled{n}$, indicate $n\in F(s)$ by Lemma $\ref{computationresults}(\ref{inFs})$. Those shaded entries indicate $n\in F(s)$ by Lemma \ref{inductionstep}. Since $s_2^*(3)=19$ by Theorem \ref{kelly}, $3\in F(s)$ for all $s\geq19$. This fact is indicated by those boxed entries, i.e., $\boxed{n}$. Finally, those plain entries indicate $n\notin F(s)$ by Lemma \ref{impossiblen}.
\end{proof}

\section{Product-admissible numbers}\label{secadmisprod}
In Section~\ref{secf(s)}, we fixed the sum in the triple $(s,p,n)$ to study the function $f(s)$.  We now turn our attention to fixing the product of the triple.

\begin{theorem}\label{admissibleprod}
Let $q_1,q_2,\dotsc,q_k$ be primes, and let $j_1,j_2,\dotsc,j_k\in\N$. Then $p=q_1^{j_1}q_2^{j_2}\dotsb q_k^{j_k}$ is product-admissible if and only if there exists a nonzero multivariate polynomial $\chi$ of $k$ variables with integer coefficients such that
\begin{itemize}
\item $\chi(q_1,q_2,\dotsc,q_k)=0$,
\item $\chi_\ell(1,1,\dotsc,1)=0$ for each $1\leq\ell\leq k$, where $\chi_\ell$ is the partial derivative of $\chi$ with respect to the $\ell$-th variable,
\item for each $1\leq\ell\leq k$, the sum of the absolute values of the coefficients in $\chi_\ell$ is at most $2j_\ell$, and
\item $\chi(1,1,\dotsc,1)=0$.
\end{itemize}
\end{theorem}

\begin{proof}
Let $q_1,q_2,\dotsc,q_k$ be primes, and let $j_1,j_2,\dotsc,j_k\in\N$ be such that $p=q_1^{j_1}q_2^{j_2}\dotsb q_k^{j_k}$ is product-admissible, i.e., there exists $n\in\N$ and at least two different multisets of $n$ positive integers, $\{x_1,x_2,\dotsc,x_n\}$ and $\{y_1,y_2,\dotsc,y_n\}$, satisfying $T\{x_1,x_2,\dotsc,x_n\}=T\{y_1,y_2,\dotsc,y_n\}=(s,p,n)$ for some $s\in\N$. Since $x_1x_2\dotsb x_n=y_1y_2\dotsb y_n=q_1^{j_1}q_2^{j_2}\dotsb q_k^{j_k}$, by the fundamental theorem of arithmetic, for each $1\leq i\leq n$, we can let $x_i=q_1^{\alpha_{i1}}q_2^{\alpha_{i2}}\dotsb q_k^{\alpha_{ik}}$ and $y_i=q_1^{\beta_{i1}}q_2^{\beta_{i2}}\dotsb q_k^{\beta_{ik}}$, where $\alpha_{i\ell},\beta_{i\ell}\in\N\cup\{0\}$ for all $1\leq i\leq n$ and $1\leq\ell\leq k$, and $\sum_{i=1}^n\alpha_{i\ell}=\sum_{i=1}^n\beta_{i\ell}=j_\ell$ for each $1\leq\ell\leq k$.

For all $(t_1,t_2,\dotsc,t_k)$ satisfying $0\leq t_\ell\leq j_\ell$ for each $1\leq\ell\leq k$, let $a_{t_1,t_2,\dotsc,t_k}$ be the number of times $q_1^{t_1}q_2^{t_2}\dotsb q_k^{t_k}$ appears in $\{x_1,x_2,\dotsc,x_n\}$, and let $b_{t_1,t_2,\dotsc,t_k}$ be the number of times $q_1^{t_1}q_2^{t_2}\dotsb q_k^{t_k}$ appears in $\{y_1,y_2,\dotsc,y_n\}$. For all $j\in\N$, let $[j]=\{0,1,2,\dotsc,j\}$. Then
\begin{itemize}
\item $\displaystyle\sum_{(t_1,t_2,\dotsc,t_k)\in[j_1]\times[j_2]\times\dotsb\times[j_k]}a_{t_1,t_2,\dotsc,t_k}q_1^{t_1}q_2^{t_2}\dotsb q_k^{t_k}\\=\sum_{(t_1,t_2,\dotsc,t_k)\in[j_1]\times[j_2]\times\dotsb\times[j_k]}b_{t_1,t_2,\dotsc,t_k}q_1^{t_1}q_2^{t_2}\dotsb q_k^{t_k}=s$,
\item $\displaystyle\sum_{(t_1,t_2,\dotsc,t_k)\in[j_1]\times[j_2]\times\dotsb\times[j_k]}a_{t_1,t_2,\dotsc,t_k}t_\ell=\sum_{(t_1,t_2,\dotsc,t_k)\in[j_1]\times[j_2]\times\dotsb\times[j_k]}b_{t_1,t_2,\dotsc,t_k}t_\ell=j_\ell$ for each $1\leq\ell\leq k$, and
\item $\displaystyle\sum_{(t_1,t_2,\dotsc,t_k)\in[j_1]\times[j_2]\times\dotsb\times[j_k]}a_{t_1,t_2,\dotsc,t_k}=\sum_{(t_1,t_2,\dotsc,t_k)\in[j_1]\times[j_2]\times\dotsb\times[j_k]}b_{t_1,t_2,\dotsc,t_k}=n$.
\end{itemize}
If we subtract the right hand side from the left, and relabel $c_{t_1,t_2,\dotsc,t_k}=a_{t_1,t_2,\dotsc,t_k}-b_{t_1,t_2,\dotsc,t_k}$ for each $(t_1,t_2,\dotsc,t_k)\in[j_1]\times[j_2]\times\dotsb\times[j_k]$, we get
\begin{itemize}
\item $\displaystyle\sum_{(t_1,t_2,\dotsc,t_k)\in[j_1]\times[j_2]\times\dotsb\times[j_k]}c_{t_1,t_2,\dotsc,t_k}q_1^{t_1}q_2^{t_2}\dotsb q_k^{t_k}=0$,
\item $\displaystyle\sum_{(t_1,t_2,\dotsc,t_k)\in[j_1]\times[j_2]\times\dotsb\times[j_k]}c_{t_1,t_2,\dotsc,t_k}t_\ell=0$ for each $1\leq\ell\leq k$,
\item $\displaystyle\sum_{(t_1,t_2,\dotsc,t_k)\in[j_1]\times[j_2]\times\dotsb\times[j_k]}|c_{t_1,t_2,\dotsc,t_k}t_\ell|\leq\displaystyle\sum_{(t_1,t_2,\dotsc,t_k)\in[j_1]\times[j_2]\times\dotsb\times[j_k]}|a_{t_1,t_2,\dotsc,t_k}t_\ell|+|b_{t_1,t_2,\dotsc,t_k}t_\ell|=2j_\ell$ for each $1\leq\ell\leq k$, and
\item $\displaystyle\sum_{(t_1,t_2,\dotsc,t_k)\in[j_1]\times[j_2]\times\dotsb\times[j_k]}c_{t_1,t_2,\dotsc,t_k}=0$.
\end{itemize}
This is equivalent to the existence of a multivariate polynomial $\chi$ of $k$ variables with integer coefficients subject to the conditions in the statement of the theorem.

Conversely, if such a multivariate polynomial $\chi\in\Z[z_1,z_2,\dotsc,z_k]$ exists, denote the coefficient of $z_1^{t_1}z_2^{t_2}\dotsb z_k^{t_k}$ by $a_{t_1,t_2,\dotsc,t_k}$ if it is positive, and denote the absolute value of the coefficient by $b_{t_1,t_2,\dotsc,t_k}$ if it is negative. For each $(t_1,t_2,\dotsc,t_k)\in[j_1]\times[j_2]\times\dotsb\times[j_k]$, let $a_{t_1,t_2,\dotsc,t_k}$ and $b_{t_1,t_2,\dotsc,t_k}$ be the number of times that $q_1^{t_1}q_2^{t_2}\dotsc q_k^{t_k}$ appears in the multisets $X$ and $Y$ respectively. Furthermore, for each $1\leq\ell\leq k$, let $j_\ell'=j_\ell-\displaystyle\sum_{(t_1,t_2,\dotsc,t_k)\in[j_1]\times[j_2]\times\dotsb\times[j_k]}a_{t_1,t_2,\dotsc,t_k}t_\ell$, and insert one copy of $q_1^{j_1'}q_2^{j_2'}\dotsb q_k^{j_k'}$ in both $X$ and $Y$. From our constructions, it is apparent that $p=q_1^{j_1}q_2^{j_2}\dotsb q_k^{j_k}$ is product-admissible with $X$ and $Y$ being the two different multisets.
\end{proof}

If we restrict to $k=1$, then Theorem \ref{admissibleprod} implies that $p=q^j$ is product-admissible if and only if there exists a nonzero polynomial $\chi$ with integer coefficients such that
\begin{itemize}
\item $\chi(q)=0$,
\item $\chi'(1)=0$, where $\chi'$ is the derivative of $\chi$,
\item the sum of the absolute values of the coefficients in $\chi'$ is at most $2j$, and
\item $\chi(1)=0$.
\end{itemize}
In other words, there exists a nonzero polynomial $\psi$ with integer coefficients such that
$$\chi(z)=(z-q)(z-1)^2\psi(z),$$
and the sum of the absolute values of the coefficients in $\chi'$ is at most $2j$. This is a nice characterization, but we go one step further and prove the following theorem.

\begin{theorem}\label{j>=2q+4}
Let $q$ be a prime and let $j\in\N$. Then $p=q^j$ is product-admissible if and only if $j\geq2q+4$.
\end{theorem}

\begin{proof}
If $\psi$ is a constant polynomial such that $\psi(z)=1$, then $\chi(z)=z^3-(q+2)z^2+(2q+1)z-q$. This implies the two multisets can be
\begin{center}
$\{q^3,\underset{2q+1\text{ copies}}{\underbrace{q,q,\dotsc,q}}\}$ and $\{\underset{q+2\text{ copies}}{\underbrace{q^2,q^2,\dotsc,q^2}},\underset{q\text{ copies}}{\underbrace{1,1,\dotsc,1}}\}$.
\end{center}
At this moment, $(s,p,n)=(q^3+2q^2+q,q^{2q+4},2q+2)$. For all $j=2q+4+j'$ for some $j'\in\N$, the two multisets can be
\begin{center}
$\{q^3,\underset{2q+1\text{ copies}}{\underbrace{q,q,\dotsc,q}},q^{j'}\}$ and $\{\underset{q+2\text{ copies}}{\underbrace{q^2,q^2,\dotsc,q^2}},\underset{q\text{ copies}}{\underbrace{1,1,\dotsc,1}},q^{j'}\}$.
\end{center}
This implies the ``if" direction of this theorem.

By computer exhaustion, we check that $p=2^j$ is not product-admissible if $1\leq j\leq7$, and $p=3^j$ is not product-admissible if $1\leq j\leq 9$, which implies the ``only if" direction for $q=2$ or $3$. As for primes $q\geq5$, we proceed as follows.

Let $m$ be the degree of $\psi$, and let $\psi(z)=c_mz^m+c_{m-1}z^{m-1}+\dotsb+c_1z+c_0\in\Z[z]$. Without loss of generality, assume that $c_m>0$. Let $\chi'(z)=\big((z-q)(z-1)^2\psi(z)\big)'=d_{m+2}z^{m+2}+d_{m+1}z^{m+1}+d_mz^m+\dotsb+d_1z+d_0\in\Z[z]$. From our constructions, for a fixed polynomial $\psi$, the lowest possible value of $j$ such that $p=q^j$ is product-admissible is given by
$$\sum_{\substack{0\leq i\leq m+2\\\text{and }d_i>0}}d_i=\sum_{\substack{0\leq i\leq m+2\\\text{and }d_i<0}}-d_i=\frac{1}{2}\sum_{i=0}^{m+2}|d_i|.$$
If $m=0$, then it is clear that the lowest possible value of $j$ is $2q+4$, attained when $c_0=1$. Consider $m>0$. Assume the contrary that for some prime $q$, there exists $j<2q+4$ such that $p=q^j$ is product-admissible. In other words,
\begin{equation}\label{boundsdi}
\sum_{\substack{0\leq i\leq m+2\\\text{and }d_i>0}}d_i=\sum_{\substack{0\leq i\leq m+2\\\text{and }d_i<0}}-d_i\leq2q+3,
\end{equation}
and in particular, $|d_i|\leq2q+3$ for all $0\leq i\leq m+2$.

Since both polynomial multiplication and differentiation are linear operators, we can describe the relationship between $c_i$ and $d_i$ with the following matrix multiplications:
{\footnotesize\begin{align*}
\begin{pmatrix}
d_{m+2}\\
d_{m+1}\\
d_{m}\\
\vdots\\
d_2\\
d_1\\
d_0
\end{pmatrix}=&\hspace{2pt}\begin{pmatrix}
m+3\\
&m+2\\
&&m+1\\
&&&\ddots\\
&&&&3\\
&&&&&2\\
&&&&&&1
\end{pmatrix}\cdot\\
&\hspace{2pt}\begin{pmatrix}
1\\
-(q+2)&1\\
2q+1&-(q+2)&1\\
-q&2q+1&-(q+2)&1\\
&-q&2q+1&-(q+2)&1\\
&&\ddots&\ddots&\ddots&\ddots\\
&&&-q&2q+1&-(q+2)&1\\
&&&&-q&2q+1&-(q+2)&1
\end{pmatrix}\begin{pmatrix}
c_m\\
c_{m-1}\\
c_{m-2}\\
\vdots\\
c_1\\
c_0\\
0\\
0
\end{pmatrix}.
\end{align*}}
Inverting the matrices to the other side, we have
\begin{align*}
\begin{pmatrix}
c_m\\
c_{m-1}\\
c_{m-2}\\
\vdots\\
c_1\\
c_0\\
0\\
0
\end{pmatrix}=&\hspace{2pt}\begin{pmatrix}
Q_0\\
Q_1&Q_0\\
Q_2&Q_1&Q_0\\
Q_3&Q_2&Q_1&Q_0\\
Q_4&Q_3&Q_2&Q_1&Q_0\\
\vdots&\ddots&\ddots&\ddots&\ddots&\ddots\\
Q_{m+1}&\ddots&Q_4&Q_3&Q_2&Q_1&Q_0\\
Q_{m+2}&Q_{m+1}&\dotsb&Q_4&Q_3&Q_2&Q_1&Q_0
\end{pmatrix}\cdot\\
&\hspace{2pt}\begin{pmatrix}
\frac{1}{m+3}\\
&\frac{1}{m+2}\\
&&\frac{1}{m+1}\\
&&&\ddots\\
&&&&\frac{1}{3}\\
&&&&&\frac{1}{2}\\
&&&&&&1
\end{pmatrix}\begin{pmatrix}
d_{m+2}\\
d_{m+1}\\
d_{m}\\
\vdots\\
d_2\\
d_1\\
d_0
\end{pmatrix},
\end{align*}
where $Q_\iota=\sum_{i=0}^\iota(\iota+1-i)q^i$ for all $\iota=0,1,2,\dotsc,m+2$.

From the second last row of the matrix multiplication, we have
$$0=\sum_{\iota=0}^{m+1}\frac{Q_\iota}{\iota+2}d_{\iota+1},$$
which implies
\begin{equation}\label{-d1}
-d_1=2\sum_{\iota=1}^{m+1}\frac{Q_\iota}{\iota+2}d_{\iota+1}.
\end{equation}
Note that for all $0\leq\iota\leq m+1$,
\begin{align*}
\frac{Q_\iota}{\iota+2}-\frac{Q_{\iota-1}}{\iota+1}&=\frac{1}{\iota+2}\left(\sum_{i=0}^\iota(\iota+1-i)q^i-\left(1+\frac{1}{\iota+1}\right)\sum_{i=0}^{\iota-1}(\iota-i)q^i\right)\\
&=\frac{1}{\iota+2}\sum_{i=0}^\iota\left(1-\frac{\iota-i}{\iota+1}\right)q^i>0.
\end{align*}
Hence, $\frac{Q_\iota}{\iota+2}$ decreases with $\iota$. From the first row of the matrix multiplication, we note that $d_{m+2}=c_m(m+3)\geq m+3$. Combining with inequality \eqref{boundsdi}, equation \eqref{-d1} becomes
\begin{align*}
-d_1&\geq2\left(\frac{Q_{m+1}}{m+3}(m+3)+\frac{Q_m}{m+2}\big(-(2q+3)\big)\right)\\
&=2\left(\sum_{i=0}^{m+1}(m+2-i)q^i-\frac{1}{m+2}\left(2\sum_{i=0}^m(m+1-i)q^{i+1}+3\sum_{i=0}^m(m+1-i)q^i\right)\right)\\
&=2\left(\sum_{i=0}^{m+1}(m+2-i)q^i-\frac{1}{m+2}\left(2\sum_{i=1}^{m+1}(m+2-i)q^i+3\sum_{i=0}^{m+1}(m+1-i)q^i\right)\right)\\
&=2\left(\sum_{i=1}^{m+1}\left(m-3-i+\frac{3+5i}{m+2}\right)q^i+\frac{m^2+m+1}{m+2}\right).
\end{align*}
It suffices to show that $-d_1>2q+4$, since this will contradict with inequality \eqref{boundsdi}.

If $m=1$, then
$$-d_1\geq\frac{2}{3}q^2-\frac{2}{3}q+2,$$
which is greater than $2q+4$ since $q\geq5$. If $m=2$, then
$$-d_1\geq q^3+\frac{1}{2}q^2+\frac{7}{2}>2q+\frac{1}{2}+\frac{7}{2}=2q+4.$$
If $m=3$, then
$$-d_1\geq\frac{6}{5}q^4+\frac{6}{5}q^3+\frac{6}{5}q^2+\frac{6}{5}q+\frac{26}{5}>\frac{24}{5}q+\frac{26}{5}>2q+4.$$
Finally, if $m\geq4$, then
\begin{align*}
&\hspace{2pt}-d_1-(2q+4)\\
\geq&\hspace{2pt}2\left(\frac{mq^{m+1}+(2m-3)q^m+(3m-6)q^{m-1}+(4m-9)q^{m-2}}{m+2}\right.\\
&\hspace{2pt}\left.+\sum_{i=2}^{m-3}\left(m-3-i+\frac{3+5i}{m+2}\right)q^i+\left(m-4+\frac{8}{m+2}-1\right)q+\frac{m^2+m+1}{m+2}-2\right).
\end{align*}
All coefficients of $q^i$ and the constant term are positive, meaning $-d_1>2q+4$.
\end{proof}

\begin{corollary}
Let $q$ be a prime and let $u\in\N$. Then $p=q^{2q+4}u$ is product-admissible.
\end{corollary}

\begin{proof}
This is by noticing that $T\{q^3,\underset{2q+1\text{ copies}}{\underbrace{q,q,\dotsc,q}},u\}=T\{\underset{q+2\text{ copies}}{\underbrace{q^2,q^2,\dotsc,q^2}},\underset{q\text{ copies}}{\underbrace{1,1,\dotsc,1}},u\}=(q^3+2q^2+q+u,q^{2q+4}u,2q+3)$.
\end{proof}

\section{At least $r$ partitions with the same product}\label{secmultipleparts}

In Sections \ref{secf(s)} and \ref{secadmisprod}, we focused on finding at least two different multisets $\{x_1,x_2,\dotsc,x_n\}$ and $\{y_1,y_2,\dotsc,y_n\}$ such that $T\{x_1,x_2,\dotsc,x_n\}=T\{y_1,y_2,\dotsc,y_n\}=(s,p,n)$. In this section, we consider at least $r$ multisets that correspond to the same triple $(s,p,n)$.

For all integers $r\geq2$ and $n\geq3$, recall from the introduction that $s_r^*(n)$ is the smallest positive integer such that for all integers $s\geq s_r^*(n)$, there are at least $r$ different $n$-partitions of $s$, namely $X_i=\{x_{i1},x_{i2},\dotsc,x_{in}\}$ for $i=1,2,\dotsc,r$, satisfying
$$T(X_i)=(s,p,n)$$
for some $p\in\N$. As mentioned in Theorem \ref{kelly}, Kelly proved that $s_{n-1}^*(n)\in\N$ exists for all integers $n\geq3$. He also stated that $s_2^*(3)=19$, but $s_{n-1}^*(n)$ was unknown for $n\geq4$.

To find the values of $s_{n-1}^*(n)$, we first define $s_r^0(n)$ as the smallest positive integer $s$ such that there are at least $r$ different $n$-partitions of $s$, namely $X_i=\{x_{i1},x_{i2},\dotsc,x_{in}\}$ for $i=1,2,\dotsc,r$, satisfying $T(X_i)=(s,p,n)$ for some $p\in\N$.

\begin{theorem}\label{s^0tos^*}
For all integers $r\geq2$ and $n\geq3$, $s_r^*(n+1)\leq s_r^0(n)+1\leq s_r^*(n)+1$.
\end{theorem}

\begin{proof}
Let $X_i=\{x_{i1},x_{i2},\dotsc,x_{in}\}$ for $i=1,2,\dotsc,r$ be $r$ different partitions of $s=s_r^0(n)$ satisfying $T(X_i)=(s,p,n)$. For any $s'\geq s_r^0(n)+1$, let $u=s'-s_r^0(n)$. Then $X'_i=\{x_{i1},x_{i2},\dotsc,x_{in},u\}$ for $i=1,2,\dotsc,r$ are $r$ different partitions of $s'$ satisfying $T(X_i)=(s',pu,n+1)$. Therefore, $s_r^*(n+1)\leq s_r^0(n)+1$. The second inequality follows from the obvious fact that $s_r^0(n)\leq s_r^*(n)$.
\end{proof}

Theorem \ref{s^0tos^*} can be used as an algorithm to determine $s_r^*(n)$ by first computing $s_r^0(n-1)$, followed by checking all values $s\leq s_r^0(n-1)+1$. To illustrate this process, we have computed $s_n^0(n)$ for $3\leq n\leq20$, listed in the following table. These results can be verified computationally.
$$\small\begin{tabular}{|c|c|c|c|c|c|c|c|c|c|c|c|c|c|c|c|c|c|c|}
\hline
$n$&3&4&5&6&7&8&9&10&11&12&13&14&15&16&17&18&19&20\\
\hline
$s_n^0(n)$&39&24&25&26&28&30&31&34&35&37&39&41&43&44&46&48&49&51\\
\hline
\end{tabular}$$
To determine $s_{n-1}^*(n)$ for $3\leq n\leq21$, we only need to check all values $s\leq s_{n-1}^0(n-1)+1$. A longer list of $s_{n-1}^*(n)$ values can be found on the On-Line Encyclopedia of Integer Sequences as A317254 \cite{oeis}.
$$\small\begin{tabular}{|c|c|c|c|c|c|c|c|c|c|c|c|c|c|c|c|c|c|c|c|}
\hline
$n$&3&4&5&6&7&8&9&10&11&12&13&14&15&16&17&18&19&20&21\\
\hline
$s_{n-1}^*(n)$&19&23&23&26&27&29&31&32&35&36&38&40&42&44&45&47&49&50&52\\
\hline
\end{tabular}$$

\section{Concluding remarks and conjectures}\label{conclusion}

Based on computational data for $6\leq n\leq60$, it can be observed that $s_{n-1}^*(n)=s_{n-1}^0(n-1)+1$, which motivates the following conjecture.

\begin{conjecture}
For all integers $n\geq6$, $s_{n-1}^*(n)=s_{n-1}^0(n-1)+1$.
\end{conjecture}

Computational data also leads us to the following conjectures. Note that in each of the following statements, $s_r^*(n)\leq s_r^*(n-1)+1$ is given by Theorem \ref{s^0tos^*}.

\begin{conjecture}\ 
\begin{enumerate}[$(a)$]
\item For all integers $n\geq9$, $s_{n-2}^*(n)=s_{n-2}^*(n-1)+1$.
\item For all integers $n\geq7$, $s_{n-1}^*(n)=s_{n-1}^*(n-1)+1$.
\item For all integers $n\geq10$, $s_n^*(n)=s_n^*(n-1)+1$.
\end{enumerate}
\end{conjecture}

\section{Acknowledgement}

This project is supported by the National Science Foundation (grant number: 1560019) through the Research Experiences for Undergraduates at Muhlenberg College in summer 2018.

\end{document}